\documentclass{amsproc}
\usepackage{amsmath}
\usepackage{amsfonts}
\usepackage{amssymb}
\usepackage{hyperref}
\hypersetup{colorlinks,
citecolor=black,
filecolor=black,
linkcolor=black,
urlcolor=black}
\usepackage{graphics,epstopdf}
\usepackage[pdftex]{graphicx}
\usepackage{newlfont}\newlength{\defbaselineskip}
\usepackage[left=1in,right=1in,top=1in,bottom=0.8in,footskip=0.25in]{geometry}
\usepackage{setspace}
\allowdisplaybreaks
\newtheorem{theorem}{Theorem}[section]
\newtheorem{example}{Example}[section]
\newtheorem{lemma}{Lemma}[section]

\newtheorem{remark}{Remark}[section]

\numberwithin{equation}{section}

\newtheorem{corollary}{Corollary}[section]

\usepackage{cite}

\usepackage{fancyhdr}
\usepackage{graphics}
\usepackage{color,colortbl}
\usepackage{rotating}
\usepackage{fancybox}
\usepackage[table]{xcolor}

\begin{document}
\title{Approximation on Durrmeyer modification of generalized Sz$\acute{\text{a}}$sz-Mirakjan operators
}
\maketitle
\begin{center}
{\bf Rishikesh Yadav$^{1,\dag}$,  Ramakanta Meher$^{1,\star}$,  Vishnu Narayan Mishra$^{2,\circledast}$}\\
$^{1}$Applied Mathematics and Humanities Department,
Sardar Vallabhbhai National Institute of Technology Surat, Surat-395 007 (Gujarat), India.\\
$^{2}$Department of Mathematics, Indira Gandhi National Tribal University, Lalpur, Amarkantak-484 887, Anuppur, Madhya Pradesh, India\\
\end{center}
\begin{center}
$^\dag$rishikesh2506@gmail.com,  $^\star$meher\_ramakanta@yahoo.com,
 $^\circledast$vishnunarayanmishra@gmail.com
\end{center}

\vskip0.5in

\begin{abstract}
This paper deals with the approximations of Durrmeyer type generalization of Sz$\acute{\text{a}}$sz-Mirakjan operators. We establish the direct results, quantitative Voronovskaya type theorem, Gr$\ddot{\text{u}}$ss type theorem, $A$-statistical convergence, rate of convergence in terms of the function with derivative of bounded variation. At last, the graphical analysis, comparison study and numerical representations of proposed operators are discussed.


\end{abstract}
\subjclass \textbf{MSC 2010}: {41A25, 41A35, 41A36}.



\section{Introduction}
In 1950, Sz$\acute{\text{a}}$sz \cite{OS} studied the approximation properties of generalization of Bernstein's operators on infinite interval and known as Sz$\acute{\text{a}}$sz-Mirakjan operators. After two decades, in 1977,  Jain and Pethe \cite{JP}, introduced new type of Sz$\acute{\text{a}}$sz-Mirakjan operators which are as follows:

\begin{eqnarray}\label{no1}
\mathcal{LO}_{n}^{[\alpha]}=\sum\limits_{i=0}^{\infty}(1+n\alpha)^{\frac{-x}{\alpha}}\left(\alpha+\frac{1}{n}\right)^{-i}\frac{x^{(i,-\alpha)}}{i!}f\left(\frac{i}{n}  \right),
\end{eqnarray}
where $x^{(i,-\alpha)}=x(x+\alpha)\cdots(x+(i-1)\alpha$, $x^{(0,-\alpha)}=1$ and the function $f$ is considered to be of exponential type such that $f(x)\leq C e^{Ax},~(x\geq 0$, $A>0)$, with positive constant $C$ and here $\alpha=\alpha_n,~n\in\mathbb{N}$ is  as $0\leq\alpha_n\leq\frac{1}{n}$. They determined the approximation properties of the said operators. For $\alpha=\frac{1}{n}$, the above operators (\ref{no1}) reduce to  the operators which have been defined by Agratini \cite{AO2}.
%
In 2007,  Abel and Ivan \cite{AI} replaced $\alpha$ by $\frac{1}{nc}$ in the above operators (\ref{no1}) and  obtained the generalized version operators, which are as:

\begin{eqnarray}\label{no2}
\mathcal{AO}_{n}^{c}=\sum\limits_{i=0}^{\infty}\left(\frac{c+1}{c}\right)^{-xnc} (1+c)^{-i} \binom{ncx+i-1}{i} f\left(\frac{i}{n}  \right),
\end{eqnarray}
where $c=c_n$, $n\in\mathbb{N}$ is restricted with a certain constant $\beta>0$ such that $c\geq \beta$. 
The main purpose to define the  above operators (\ref{no2}), was to investigate the local approximation properties  and to check the asymptotic behavior. Very recently, Dhamiza et al. \cite{DPD} defined Kantorovich variant of the operators (\ref{no1}), for the studying the local approximations properties and rate of convergence. For bounded and integrable function on $[0,\infty)$, the operators are defined by  
 \begin{eqnarray}
\mathcal{MO}_{n}^{\alpha}=\sum\limits_{i=0}^{\infty}(1+n\alpha)^{\frac{-x}{\alpha}}\left(\alpha+\frac{1}{n}\right)^{-i}\frac{x^{(i,-\alpha)}}{i!}\int\limits_{\frac{i}{n}}^{\frac{i+1}{n}} f(t)dt.
\end{eqnarray}
A special case, when $\alpha\to 0$, the above operators reduce to Sz$\acute{\text{a}}$sz-Mirakjan-Kantorovich  operators defined by Totik  \cite{vt}. The property of Kantorovich type operators are also discussed in \cite{RYRVN,RYMVN}. But for the Durremnyer point of view of the Sz$\acute{\text{a}}$sz-Mirakjan operators, in 1985, Mazhar and Totik \cite{SV}, modified the Sz$\acute{\text{a}}$sz-Mirakjan operators into summation integral type operators, which are defined by
\begin{eqnarray}\label{d1}
I_{n}(h;x)=n\sum\limits_{i=0}^{\infty}u_{n,i}(x)\int\limits_0^\infty u_{n,i}(t)h(t)~dt,
\end{eqnarray}
where $u_{n,i}(x)=e^{-nx}\frac{(nx)^i}{i!}$, and independently the related aaproximations properties have been discussed by Kasana et al. \cite{HSG}. In this regard, Gupta and Pant \cite{VRP} determined the rate of convergence and other approximations properties. Some approximations properties can be seen in \cite{MKMS}. Their modifications into Durrmeyer version and properties can be seen in various research articles, such as \cite{GKR,VGA,VG,KAP}. In 2016, Mishra et al. \cite{VNM} modified Sz$\acute{\text{a}}$sz-Mirakjan Durrmeyer operators using a positive sequence of functions to study the properties like simultaneous approximation, rate of convergence etc. The modified operators are as
\begin{eqnarray}\label{d2}
D_{n}(h;x)=d_n\sum\limits_{i=0}^{\infty}u_{d_n,i}(x)\int\limits_0^\infty u_{d_n,i}(t)h(t)~dt,
\end{eqnarray}
where, $d_n\to\infty$ as $n\to\infty$ be a positive sequence of real number which is strictly increasing as well as $d_1\geq 1$ and $u_{d_n,i}(x)=e^{-d_nx}\frac{(d_nx)^i}{i!}$. Moreover, the operators (\ref{d1}) can be obtained when $d_n=n$ in the the operators (\ref{d2}).




Motivated by above works, we define the Durrmeyer modification of the above operators (\ref{no1}) by considering the function as  integrable and bounded on the interval $[0,\infty)$ as follows:


\begin{eqnarray}\label{O1}
\mathcal{U}_{n}^{[\alpha]}(f;x)=n\sum\limits_{i=0}^{\infty}(1+n\alpha)^{\frac{-x}{\alpha}}\left(\alpha+\frac{1}{n}\right)^{-i}\frac{x^{(i,-\alpha)}}{i!}\int\limits_{0}^{\infty}e^{-n u}\frac{(n u)^i}{i!} f(u)~du.
\end{eqnarray}
For $\alpha\to 0$, the above operators will  be reduced into Sz$\acute{\text{a}}$sz-Mirakjan-Durrmeyer operators, which are defined by equation (\ref{d1}).\\
The main motive of this article is to investigate the approximation properties of the defined operators (\ref{O1}). Therefore, we divide it into sections. The rate of convergence of the defined operators is obtained in the terms of modulus of continuity, second order modulus of continuity with the relations of  Peetre's $K$-functionals in the section  \ref{sec2}. Section \ref{sec3} consists, weighted approximations properties including convergence of the operators using weight function. To determine the properties of the operators (\ref{O1}) via quantitatively, quantitative Voronovskaya type theorem and  Gr$\ddot{\text{u}}$ss Voronovskaya type theorem are studied in section \ref{sec4}. In section \ref{sec5}, graphical and numerical representations are presented for the support of approximation results. Section \ref{sec6} represent $A$-statistical properties of the operators and in section (\ref{sec7}), an important property is studied for the rate of the convergence by means of the derivative of bounded variations. Finally, conclusion, result discussion and  applications are discussed. 

\section{Preliminaries}
This section contains, basic lemmas, remark and theorem, which are used to prove our main theorems and study the approximations properties of the proposed operators. Here, we need following lemma.

\begin{lemma}\label{l2}
Following results hold for all $n\in \mathbb{N}$:
\begin{eqnarray*}
\mathcal{U}_{n}^{[\alpha]}(1;x)&=& 1\\
\mathcal{U}_{n}^{[\alpha]}(t;x)&=& \frac{1+nx}{n}\\
\mathcal{U}_{n}^{[\alpha]}(t^2;x)&=& \frac{2+4nx+n^2x^2+n^2x\alpha}{n^2}\\
\mathcal{U}_{n}^{[\alpha]}(t^3;x)&=& \frac{6+18nx+9n^2x(x+\alpha)+n^3x(x^2+3x\alpha+2\alpha^2)}{n^3}\\
\mathcal{U}_{n}^{[\alpha]}(t^3;x)&=& \frac{24+96nx+72n^2x(x+\alpha)+16n^3x(x^2+3x\alpha+2\alpha^2)+n^4x(x^3+6x^2\alpha+11x\alpha^2+6\alpha^3)}{n^4}.
\end{eqnarray*}
\begin{remark}
Here
\begin{eqnarray}\label{s}
\int\limits_{0}^{\infty}e^{-n u}\frac{(n u)^i}{i!} u^m~du=\frac{1}{n^{m+1}}\frac{(i+m)!}{i!},
\end{eqnarray}
in particular, if $m=0$ then 
\begin{eqnarray*}
\int\limits_{0}^{\infty}e^{-n u}\frac{(n u)^i}{i!} ~du=\frac{1}{n},
\end{eqnarray*}
if $m=1$ then 
\begin{eqnarray*}
\int\limits_{0}^{\infty}e^{-n u}\frac{(n u)^i}{i!} u ~du=\frac{(i+1)}{n^2},
\end{eqnarray*}

and so on...
\end{remark}

\begin{remark}
Notice that
\begin{eqnarray}\label{A}
(1+n\alpha)^{\frac{x}{\alpha}}=\sum\limits_{i=0}^{\infty}\left(\alpha+\frac{1}{n}\right)^{-i}\frac{x^{(i,-\alpha)}}{i!}.
\end{eqnarray}

\end{remark}
\begin{proof}
Using the equations (\ref{s}) and (\ref{A}), we get
\begin{eqnarray*}
\mathcal{U}_{n}^{[\alpha]}(1;x)&=&\sum\limits_{i=0}^{\infty}(1+n\alpha)^{\frac{-x}{\alpha}}\left(\alpha+\frac{1}{n}\right)^{-i}\frac{x^{(i,-\alpha)}}{i!}\\
&=& (1+n\alpha)^{\frac{-x}{\alpha}}\sum\limits_{i=0}^{\infty} \left(\alpha+\frac{1}{n}\right)^{-i}\frac{x^{(i,-\alpha)}}{i!}\\
&=& (1+n\alpha)^{\frac{-x}{\alpha}} (1+n\alpha)^{\frac{x}{\alpha}}=1\\
\mathcal{U}_{n}^{[\alpha]}(u;x)&=& \sum\limits_{i=0}^{\infty}(1+n\alpha)^{\frac{-x}{\alpha}}\left(\alpha+\frac{1}{n}\right)^{-i}\frac{x^{(i,-\alpha)}}{i!} \frac{i+1}{n}\\
&=&\frac{(1+n\alpha)^{\frac{-x}{\alpha}}}{n}\sum\limits_{i=1}^{\infty}\left(\alpha+\frac{1}{n}\right)^{-i}\frac{x^{(i,-\alpha)}}{i!} i+ \frac{1}{n}\sum\limits_{i=1}^{\infty}(1+n\alpha)^{\frac{-x}{\alpha}}\left(\alpha+\frac{1}{n}\right)^{-i}\frac{x^{(i,-\alpha)}}{i!}\\
&=& x+\frac{1}{n}.
\end{eqnarray*}
\end{proof}

\end{lemma}
Consider the $\Theta_{n,m}^{[\alpha]}(x)=\mathcal{U}_{n}^{[\alpha]}((t-x)^m;x)$, $m=1,2,3$ are the central moments and here we obtain the following the results.
\begin{lemma}\label{l4}
For each $x\geq 0$ and $n\in\mathbb{N}$, it holds:
\begin{eqnarray*}
\Theta_{n,1}^{[\alpha]}(x)&=& \frac{1}{n}\\
\Theta_{n,2}^{[\alpha]}(x)&=&\frac{\alpha n^2 x+2 n x+2}{n^2} \\
\Theta_{n,3}^{[\alpha]}(x)&=& \frac{2 \alpha ^2 n^3 x+9 \alpha  n^2 x+12 n x+6}{n^3}\\
\Theta_{n,4}^{[\alpha]}(x)&=&\frac{3 \alpha ^2 n^4 x (2 \alpha +x)+4 \alpha  n^3 x (8 \alpha +3 x)+12 n^2 x (6 \alpha +x)+72 n x+24}{n^4}
\end{eqnarray*}
\end{lemma}
\begin{proof}
Using the Lemma \ref{l2}, we can easily prove all parts of the above lemma, so we omit the proof. 
\end{proof}

\begin{remark}
For $x\in[0,\infty)$ and for $n\in\mathbb{N}$, we obtain
\begin{eqnarray*}
\Theta_{n,2}^{[\alpha]}(x)&=& \frac{\alpha n^2 x+2 n x+2}{n^2}=\alpha x+\frac{2x}{n}+\frac{2}{n^2}\\
&\leq & \frac{3x}{n}+\frac{2}{n^2}=\frac{3}{n}\left(x+\frac{1}{n}\right)=\frac{3}{n}\eta_n^2(x)
\end{eqnarray*}
\end{remark}

\begin{remark}
The above operators \ref{O1} can be written as
\begin{eqnarray*}
\mathcal{U}_{n}^{[\alpha]}(f;x)=\int_0^\infty u_{n}^{[\alpha]}  (x,t)f(t)~dt
\end{eqnarray*}
where $u_{n}^{[\alpha]}(x,t)=n\sum\limits_{i=0}^\infty r_{n,i}^{[\alpha]}(x)p_n(t)$, $r_{n,i}^{[\alpha]}(x)=(1+n\alpha)^{\frac{-x}{\alpha}}\left(\alpha+\frac{1}{n}\right)^{-i}\frac{x^{(i,-\alpha)}}{i!}$ and $p_n(x)=e^{-n x}\frac{(n x)^i}{i!}$.  
\end{remark}

\begin{lemma}\label{l1}
For every $x\geq 0$ and $\max\alpha=\frac{1}{n}$ then it holds:
\begin{eqnarray*}
\underset{n\to\infty}\lim\{n\Theta_{n,1}^{\alpha}(x)\}&=& 1,\\
\underset{n\to\infty}\lim\{n\Theta_{n,2}^{\alpha}(x)\}&=& 3x,\\
\underset{n\to\infty}\lim\{n^2\Theta_{n,4}^{\alpha}(x)\}&=& 27x^2,\\
\underset{n\to\infty}\lim\{n^3\Theta_{n,6}^{\alpha}(x)\}&=& 405x^3.\\
\end{eqnarray*}
\end{lemma}

\begin{lemma}
If a function $g$ defined on $[0,\infty)$ and bounded with supremum norm $\| f\|=\underset{x\geq0}\sup |f(x)|$ then there is an inequality holds as:
\begin{eqnarray*}
\left|\mathcal{U}_{n}^{[\alpha]}(f;x) \right|\leq \|f\|.
\end{eqnarray*}
\end{lemma}
\begin{theorem}\label{th1}
Consider $g\in C_A[0,\infty)$, $\alpha\in[0,\frac{1}{n}]$ and $\alpha\to 0$ as $n\to\infty$ then we get 
\begin{eqnarray*}
\underset{n\to\infty}\lim\mathcal{U}_{n}^{[\alpha]}(g;x)=g(x),
\end{eqnarray*}
uniformly on each finite interval of $[0,\infty)$. 
\end{theorem}

\section{Direct Results}\label{sec2}
This segment, consists the uniform convergence theorem, rate of the convergence of the proposed operators using second order modulus of smoothness and modulus of continuity and a relation exist with Peetre's $K$-functional.  \\

To study the approximation properties, we suppose $C_B[0,\infty)$ be the set of all continuous and bounded function $g$ defined on $[0,\infty)$ with supremum norm $\|g\|=\underset{x\geq0} \sup |g(x)|$. And here for any $\xi>0$, the Peetre's $K$-functional is defined by:
\begin{eqnarray*}
K_2(g;\xi)=\underset{g_1\in C_B^2[0,\infty)}\inf\{\|g-g_1\|+\xi \|g_1''\|\}, ~~\text{where}~C_B^2[0,\infty)=\{g\in C_B[0,\infty):g',g''\in C_B[0,\infty)\}.
\end{eqnarray*} 
In 1993, an important relation was introduced by De Vore \cite{q1} by considering Peetre's $K$-functional and second order modulus of smoothness, which is given below:
\begin{eqnarray}\label{pe1}
K_2(g;\xi)\leq M \omega_2(f;\sqrt{\xi}),
\end{eqnarray} 
where $ \omega_2(f;\sqrt{\xi})$ is second order modulus of smoothness and is defined by, 
\begin{eqnarray*}
\omega_2(f;\sqrt{\xi})=\underset{h\in[0,\sqrt{\xi}],x\in[0,\infty)}\sup \{|f(x+2h)-2f(x+h)+f(x)|\}, ~f\in C_B[0,\infty)
\end{eqnarray*}


\textbf{Note.} The first order modulus of continuity is $\omega(f;\xi)=\underset{h\in[0,\xi],x\in[0,\infty)}\sup \{|f(x+h)-f(x)|\}, ~f\in C_B[0,\infty)$.\\
\begin{theorem}
Consider $f\in C_B[0,\infty)$ and for every $x\in[0,\infty)$, it holds as:
\begin{eqnarray*}
 \left|\mathcal{U}_{n}^{[\alpha]}(f;x)-f(x)\right|\leq M \omega_2(f;\sqrt{\rho_n(x)})+\omega\left(f;\nu_n(x) \right),
\end{eqnarray*}
where, $\nu_n(x)=\Theta_{n,1}^{[\alpha]}(x)$ and $\rho_n(x)=\Theta_{n,2}^{[\alpha]}(x)+\frac{1}{n^2}$.
\end{theorem}
\begin{proof}
For $f\in C_B[0,\infty)$, consider
\begin{eqnarray*}
\mathfrak{U}_{n}^{[\alpha]}(f;x)=\mathcal{U}_{n}^{[\alpha]}(f;x)+f(x)-f\left(\frac{1+nx}{n}\right).
\end{eqnarray*} 
Here, $\mathfrak{U}_{n}^{[\alpha]}(1;x)=1$ and $\mathfrak{U}_{n}^{[\alpha]}(t;x)=x$, i.e. the auxiliary operators preserve the linear function and constant term. Now, using Taylor's remainder formula for integral for the function $u\in C_B^2[0,\infty)$, we have
\begin{eqnarray*}
u(t)-u(x)=(t-x)u'(x)+\int\limits_x^t (t-v)u''(v)~dv.
\end{eqnarray*}
Applying the operators $\mathfrak{U}_{n}^{[\alpha]}$ to the both sides, we obtain
\begin{eqnarray*}
\mathfrak{U}_{n}^{[\alpha]}(u(t)-u(x);x)&=&\mathfrak{U}_{n}^{[\alpha]}\left(\int\limits_x^t (t-v)u''(v)~dv \right)\\
&=&  \mathcal{U}_{n}^{[\alpha]}\left(\int\limits_x^t (t-v)u''(v)~dv\right)-\left(\int\limits_x^{\left(\frac{1+nx}{n}\right)} \left(\frac{1+nx}{n}-v\right)u''(v)~dv\right).
\end{eqnarray*}
Here,
\begin{eqnarray*}
\left|\int\limits_x^t (t-v)u''(v)~dv\right|\leq \int\limits_x^t |t-v||u''(v)|~dv\leq |u''|(t-x)^2.
\end{eqnarray*}
Similarly, 
\begin{eqnarray*}
\left|\int\limits_x^{\left(\frac{1+nx}{n}\right)} \left(\frac{1+nx}{n}-v\right)u''(v)~dv \right|\leq \frac{|u''|}{n^2}.
\end{eqnarray*}
Therefore, one has
\begin{eqnarray*}
\left|\mathcal{U}_{n}^{[\alpha]}(f;x)-f(x)\right| & \leq & |u''| \Theta_{n,2}^{[\alpha]}(x)+\frac{|u''|}{n^2}\\
&=& \rho_n(x)|u''|
\end{eqnarray*}
Also,
\begin{eqnarray*}
|\mathfrak{U}_{n}^{[\alpha]}(f;x)|=|\mathcal{U}_{n}^{[\alpha]}(f;x)|+2\|f(x)\|\leq 3\|f(x)\|. 
\end{eqnarray*}
By considering all above inequalities, we obtain
\begin{eqnarray*}
\left|\mathcal{U}_{n}^{[\alpha]}(f;x)-f(x)\right|&\leq &|\mathfrak{U}_{n}^{[\alpha]}(f-u;x)|+|\mathfrak{U}_{n}^{[\alpha]}(u(t)-u(x);x)|+|u(x)-f(x)|+\left|f\left(\frac{1+nx}{n}\right)-f(x) \right|\\
&\leq & 4\|f-u\|+\rho_n(x)\|u''\|+\left|f\left(\frac{1+nx}{n}\right)-f(x) \right|\\
&\leq & M \{\|f-u\|+\rho_n(x)\|u''\|\}+\omega\left(f;\nu_n(x) \right).
\end{eqnarray*}
Now, taking minimum overall $u\in C_B^2[0,\infty)$ on the right hand side of above inequality and using Peetre $K$-functional, we obtain
\begin{eqnarray*}
\left|\mathcal{U}_{n}^{[\alpha]}(f;x)-f(x)\right| &\leq & M K_2(f;\rho_n(x))+\omega\left(f;\nu_n(x) \right)\\
&\leq & M \omega_2(f;\sqrt{\rho_n(x)})+\omega\left(f;\nu_n(x) \right).
\end{eqnarray*}
\end{proof}
Now, we estimate the approximation of the defined operators (\ref{O1}), by new type of Lipschit maximal function with order $r\in(0,1]$,  defined by Lenze \cite{LB} as
\begin{eqnarray}\label{eq8}
\kappa_r(f,x)=\underset{x,s\geq 0}\sup \frac{|f(s)-f(x)|}{|s-x|^r},~~x\neq s. 
\end{eqnarray}
Using Lipschit maximal function, we have an upper bound with the function, given by a theorem.
\begin{theorem}
Consider $f\in C_B[0,\infty)$ with $r\in(0,1]$ then we obtain
\begin{eqnarray*}
\left|\mathcal{U}_{n}^{[\alpha]}(f;x)-f(x)\right| &\leq  \kappa_r(f,x)\left(\nu_n(x)\right)^{\frac{r}{2}}.
\end{eqnarray*}
\end{theorem}
\begin{proof}
By equation (\ref{eq8}), we can write
\begin{eqnarray*}
\left|\mathcal{U}_{n}^{[\alpha]}(f;x)-f(x)\right| &\leq \kappa_r(f,x)\mathcal{U}_{n}^{[\alpha]}(|s-x|^r;x).
\end{eqnarray*}
Using, H$\ddot{\text{o}}$lder's inequality with $j=\frac{2}{r}$, $l=\frac{2}{2-r}$, one can get
\begin{eqnarray*}
\left|\mathcal{U}_{n}^{[\alpha]}(f;x)-f(x)\right| &\leq & \kappa_r(f,x)\left(\mathcal{U}_{n}^{[\alpha]}((s-x)^2;x)\right)^{\frac{r}{2}}=\kappa_r(f,x)\left(\nu_n(x)\right)^{\frac{r}{2}}.
\end{eqnarray*}  
Hence proved.
\end{proof}
Next theorem is based on modified Lipschitz type spaces \cite{OMA} and this spaces is defined by 
\begin{eqnarray*}
Lip_M^{\lambda_1,\lambda_2}(s)=\Bigg\{ f\in C_B[0,\infty):|f(j)-f(k)|\leq M\frac{|j-k|^s}{\left(j+k^2\lambda_1+k\lambda_2\right)^{\frac{s}{2}}},~~\text{where}~j,k\geq0 ~\text{are~variables},~s\in(0,1] \Bigg\}
\end{eqnarray*}
and $\lambda_1, \lambda_1$ are the fixed numbers.
\begin{theorem}
For $f\in Lip_M^{\lambda_1,\lambda_2}(s)$ and $0<s\leq 1$, we have an inequality holds:
\begin{eqnarray*}
\left|\mathcal{U}_{n}^{[\alpha]}(f;x)-f(x)\right| &\leq & M\left(\frac{\Theta_{n,2}^{[\alpha]}(x)}{x(x\lambda_1+\lambda_2)}\right)^{\frac{s}{2}}.
\end{eqnarray*}
\end{theorem}
\begin{proof}
To prove the above theorem, we can distribute its proof into two part by considering the case discussion. So here:\\
\textbf{Case 1.} if $s=1$, proceed ahead, we can observe that  $\frac{1}{(y+x^2\lambda_1+x\lambda_2)}\leq \frac{1}{x(x\lambda_1+\lambda_2)}$ then one has
\begin{eqnarray*}
\left|\mathcal{U}_{n}^{[\alpha]}(f;x)-f(x)\right| &\leq & \mathcal{U}_{n}^{[\alpha]}(|f(t)-f(x)|;x)\\
&\leq & M \mathcal{U}_{n}^{[\alpha]}\left(\frac{|t-x|}{\left(t+x^2\lambda_1+x\lambda_2\right)^{\frac{1}{2}}};x\right)\\
&\leq & \frac{M}{\left(x(x\lambda_1+\lambda_2)\right)^{\frac{1}{2}}}\mathcal{U}_{n}^{[\alpha]}(|t-x|;x) \\
&\leq & \frac{M}{\left(x(x\lambda_1+\lambda_2)\right)^{\frac{1}{2}}}\left(\Theta_{n,2}^{[\alpha]}(x)\right)^{\frac{1}{2}}\\
&\leq & M\left(\frac{\Theta_{n,2}^{[\alpha]}(x)}{x(x\lambda_1+\lambda_2)}\right)^{\frac{1}{2}}.
\end{eqnarray*}
\textbf{Case 2.} if $s\in (0,1)$ then with the help of H$\ddot{\text{o}}$lder inequality by considering $l=\frac{2}{s}, m=\frac{2}{2-s}$, we get
\begin{eqnarray*}
\left|\mathcal{U}_{n}^{[\alpha]}(f;x)-f(x)\right| &\leq & \left(\mathcal{U}_{n}^{[\alpha]}(|f(t)-f(x)|^{\frac{2}{s}};x)\right)^{\frac{s}{2}}\leq M\mathcal{U}_{n}^{[\alpha]}\left(\frac{|t-x|^{2}}{\left(t+x^2\lambda_1+x\lambda_2\right)};x\right)^{\frac{s}{2}}\\
&\leq & M\mathcal{U}_{n}^{[\alpha]}\left(\frac{|t-x|^{2}}{\left(x(x\lambda_1+\lambda_2)\right)};x\right)^{\frac{s}{2}}\\
&\leq & M\left(\frac{\Theta_{n,2}^{[\alpha]}(x)}{x(x\lambda_1+\lambda_2)}\right)^{\frac{s}{2}}.
\end{eqnarray*}
Thus, the proof is completed. 
\end{proof}
Let $g\in C_B[0,\infty)$ and we define Steklov mean function, which is as follows:
\begin{eqnarray*}
G_h(x)= \frac{1}{h^2}\int\limits_{-\frac{h}{2}}^{\frac{h}{2}}\int\limits_{-\frac{h}{2}}^{\frac{h}{2}} 2(g(x+\kappa+\lambda))-g(x+2(\kappa+\lambda))~d\kappa~d\lambda,~~~\kappa, \lambda\geq0~ \text{and}~h>0.
\end{eqnarray*}
To approximate continuous functions by smoother functions, Steklov function is used and we appeal to investigate the approximation properties.
So, we collect some properties in a next lemma, which are used to prove the main theorem.

\begin{lemma}
Let $g\in C_B[0\infty)$, it is holds following inequalities:

\begin{enumerate}
\item $\|G_h-g\|_{C_B[0\infty)} \leq \omega_2(g,h)$
\item $\|G_h' \|_{C_B[0\infty)} \leq \frac{5\omega(g,h)}{h}$~~~~for $G_h'\in C_B[0,\infty)$
\item $\|G_h'' \|_{C_B[0\infty)} \leq \frac{9\omega_2(g,h)}{h^2}$~~for $G_h''\in C_B[0,\infty),$
\end{enumerate}

where $\omega(g;h)$ and $\omega_2(g,h)$ are modulus of continuity and second order of modulus of continuity respectively and can be defined as another way, given by 
\begin{eqnarray*}
\omega(f,h)&=&\underset{x,\kappa,\lambda\geq 0}\sup\underset{|\kappa-\lambda|\leq h}\sup|f(x+\kappa)-f(x+\lambda)|,\\
\omega_2(f,h)&=&\underset{x,\kappa,\lambda\geq 0}\sup\underset{|\kappa-\lambda|\leq h}\sup|f(x+2\kappa)-2f(x+\kappa+\lambda)|+f(x+2\lambda).
\end{eqnarray*}
\end{lemma}
\begin{theorem}
Consider $f\in C_B[0,\infty)$, for every $x\in [0,\infty)$, we get
\begin{eqnarray*}
|\mathcal{U}_{n}^{[\alpha]}(f;x)-f(x)|\leq 5\left\{\omega\left(f,\sqrt{\Theta_{n,2}^{[\alpha]}}\right)+\frac{13}{10} \omega_2\left(f,\sqrt{\Theta_{n,2}^{[\alpha]}}\right)\right\},
\end{eqnarray*} 
where $\Theta_{n,2}^{[\alpha]}$ is defined by Lemma \ref{l4}.
\end{theorem}

\begin{proof}
For every $x\geq 0$, using Steklov function, we can write as
\begin{eqnarray*}
|\mathcal{U}_{n}^{[\alpha]}(f;x)-f(x)|\leq \mathcal{U}_{n}^{[\alpha]}(|f-G_h|;x)+|\mathcal{U}_{n}^{[\alpha]}(G_h-G_h(x);x)|+|G_h(x)-f(x)|.
\end{eqnarray*}
Since $|\mathcal{U}_{n}^{[\alpha]}(f;x)|\leq \|f(x)\|_{C_B[0,\infty)}$ as $f\in C_B[0,\infty)$ and $x\geq 0$. Then using Steklov mean property, we can have
\begin{eqnarray*}
\mathcal{U}_{n}^{[\alpha]}(|f-G_h|;x)\leq \|\mathcal{U}_{n}^{[\alpha]}(f-G_h;x)\|_{C_B[0,\infty)}\leq \|f-G_h\|_{C_B[0,\infty)}\leq \omega_2(f,h).
\end{eqnarray*}
Using the Taylor's formula and on applying the operators (\ref{O1}), we  can  write
\begin{eqnarray*}
|\mathcal{U}_{n}^{[\alpha]}(G_h-G_h(x);x)|\leq \|G_h'\|_{C_B[0,\infty)}\sqrt{\Theta_{n,2}^{[\alpha]}}+\frac{\|G_h''\|_{C_B[0,\infty)}}{2!} \Theta_{n,2}^{[\alpha]}.
\end{eqnarray*}
Using the propery of the Steklov mean, we can write as
\begin{eqnarray*}
|\mathcal{U}_{n}^{[\alpha]}(G_h-G_h(x);x)|\leq \frac{5\omega(f,h)}{h}\sqrt{\Theta_{n,2}^{[\alpha]}}+\frac{9\omega_2(f,h)}{2h^2}\Theta_{n,2}^{[\alpha]}.
\end{eqnarray*}
Choosing $h=\sqrt{\Theta_{n,2}^{[\alpha]}}$, we obtain our required result.


\end{proof}

\begin{remark}
If $\Theta_{n,2}^{[\alpha]}\to 0$ as $n\to\infty$ and then $\omega\left(f,\sqrt{\Theta_{n,2}^{[\alpha]}}\right)\to 0,~\omega_2\left(f,\sqrt{\Theta_{n,2}^{[\alpha]}}\right)\to 0$, this imply that $\mathcal{U}_{n}^{[\alpha]}(f;x)$ converge to the function $f(x)$.
\end{remark}



\section{Weighted Approximation}\label{sec3}
For describing the approximation properties,  of any sequence of linear positive operators, Gadzhiev \cite{q5,q6} introduced the weighted spaces. Recall from there, we consider the functions classes,  which are as:


$B_w[0,\infty)=\{f:[0,\infty)\to\mathbb{R} |~~  |f(x)|\leq Mw(x)~~\text{with~the~supremum~norm}~~ \|f\|_w=\underset{x\in [0,\infty)}\sup\frac{f(x)}{w(x)}<+\infty \}$, where $M>0$ is a constant depending on $f$. Also define the spaces
$$C_w[0,\infty)=\{f\in B_w[0,\infty), ~f~\text{is~contiuous} \},$$
$$C_w^k[0,\infty)=\{f\in C_w[0,\infty),\underset{x\to\infty}\lim\frac{|f(x)|}{w(x)}=k_f<+\infty\},$$      
where $w(x)=1+x^2$ is a weight function.
\begin{lemma}\cite{G1,G2}
Let $\mathcal{L}_n:C_w[0,\infty)\to B_w[0,\infty)$ with the the conditions $\underset{n\to\infty}\lim\|L_n(t^r;x)-x^r\|_w=0,~~r=0,1,2$, then for $f\in C_w^k[0,\infty)$, we have 
$$\underset{n\to\infty}\lim\|L_n(f;x)-f(x)\|_w=0.$$
\end{lemma}

\begin{theorem}
Let $\{\mathcal{U}_{n}^{[\alpha]}\}$ be a sequence defined by (\ref{O1}) then it holds as:
\begin{eqnarray*}
\underset{n\to\infty}\lim \|\mathcal{U}_{n}^{[\alpha]}(f;x)-f(x)\|_w=0,~~\text{for}~f\in C_w^k[0,\infty).
\end{eqnarray*}
\end{theorem}

If we show that $\underset{n\to\infty}\lim \|\mathcal{U}_{n}^{[\alpha]}(t^r;x)-x^r\|_w=0$ holds for $r=0,1,2$ then the above theorem will be proved. Here, it's obvious

\begin{eqnarray}
\underset{n\to\infty}\lim \|\mathcal{U}_{n}^{[\alpha]}(1;x)-1\|_w=0.
\end{eqnarray}
 Using the Lemma \ref{l2}, we have
 \begin{eqnarray*}
\|\mathcal{U}_{n}^{[\alpha]}(t;x)-x\|_w &=&\frac{1}{n}\underset{x\geq 0}\sup \frac{1}{1+x^2}\leq \frac{1}{n}\\
\Rightarrow  \|\mathcal{U}_{n}^{[\alpha]}(t;x)-x\|_w\to 0,~~~~\text{as}~n\to\infty.
 \end{eqnarray*}

 Also,
 \begin{eqnarray*}
 \|\mathcal{U}_{n}^{[\alpha]}(t^2;x)-x^2\|_w &=& \underset{x\geq 0}\sup \frac{\left|\frac{2+4nx+n^2x^2+n^2x\alpha}{n^2}-x^2\right|}{1+x^2}\\
&\leq & \frac{1}{n}\left(\frac{2}{n}\underset{x\geq 0}\sup\frac{1}{1+x^2}+5\underset{x\geq 0}\sup\frac{x}{1+x^2}\right)\\
&\leq & \frac{2}{n^2}+\frac{5}{2n}\\
\Rightarrow  \|\mathcal{U}_{n}^{[\alpha]}(t^2;x)-x^2\|_w\to 0~~~~\text{as}~n\to\infty.
 \end{eqnarray*}


And hence the proof is completed.
\begin{theorem}
For every $x\geq 0$ and $\alpha=\alpha(n)\to 0$ as $n\to\infty$, let $g\in C_w^k[0,\infty)$ and $l>0$ then we yield
\begin{eqnarray}
\underset{n\to\infty}\lim \underset{x\geq0}\sup \frac{\left|\mathcal{U}_{n}^{[\alpha]}(g;x)-g(x)\right|}{(1+x^2)^{1+l}}=0.
\end{eqnarray}
\end{theorem}
\begin{proof}
Consider $x_0$ be a fixed point, then we can right as
\begin{eqnarray}\label{eq1}
\nonumber\underset{x\geq0}\sup \frac{\left|\mathcal{U}_{n}^{[\alpha]}(g;x)-g(x)\right|}{(1+x^2)^{1+l}}&\leq & \underset{x\leq x_0}\sup \frac{\left|\mathcal{U}_{n}^{[\alpha]}(g;x)-g(x)\right|}{(1+x^2)^{1+l}}+\underset{x>x_0}\sup \frac{\left|\mathcal{U}_{n}^{[\alpha]}(g;x)-f(x)\right|}{(1+x^2)^{1+l}}\\
\nonumber &\leq & \|\mathcal{U}_{n}^{[\alpha]}(g;x)-g(x)\|+\|f\|_{w} \underset{x>x_0}\sup \frac{\left|\mathcal{U}_{n}^{[\alpha]}((1+t^2);x)\right|}{(1+x^2)^{1+l}} +\underset{x>x_0}\sup\frac{|g|}{(1+x^2)^{1+l}}\\
&=& L_1+L_2+L_3~(say). 
\end{eqnarray}
Here,
\begin{eqnarray*}
L_3=\underset{x>x_0}\sup\frac{|g|}{(1+x^2)^{1+l}}\leq \frac{\|g\|_{w}}{(1+x_0^2)^l},~~\text{(as~$|g(x)|\leq M(1+x^2))$}
\end{eqnarray*}
so, for large value of $x_0$, we can consider an arbitrary $\epsilon>0$ such that 
\begin{eqnarray}\label{eq2}
L_3=\frac{\|g\|_{w}}{(1+x_0^2)^l}\leq \frac{\epsilon}{3}.
\end{eqnarray}
Since, $\underset{n\to\infty}\lim \underset{x>x_0}\sup \frac{\left|\mathcal{U}_{n}^{[\alpha]}((1+t^2);x)\right|}{(1+x^2)}=1$, so let us consider for any arbitrary $\epsilon>0$, there exist $n_1\in\mathbb{N}$, such that 
\begin{eqnarray}\label{eq3}
L_2=\| g\|_w\underset{x>x_0}\sup \frac{\left|\mathcal{U}_{n}^{[\alpha]}((1+t^2);x)\right|}{(1+x^2)^{1+l}}\leq \frac{\| g\|_w}{(1+x^2)^l}\leq \frac{\| g\|_w}{(1+x_0^2)^l}<\frac{\epsilon}{3}.
\end{eqnarray}
Applying the theorem \ref{th1}, we can have
\begin{eqnarray}\label{eq4}
L_1=\|\mathcal{U}_{n}^{[\alpha]}(g;x)-g(x)\|_{C[0,x_0]}\|\leq \frac{\epsilon}{3}.
\end{eqnarray}
Combining (\ref{eq2}-\ref{eq4}) and using in (\ref{eq1}), we obtain our required result.
\end{proof}
\begin{theorem}
For $g\in C_w[0,\infty)$, one can obtain
\begin{eqnarray*}
|\mathcal{U}_{n}^{[\alpha]}(g;x)-g(x)|\leq 4\mathcal{N}_f(1+x^2) \Theta_{n,2}^{[\alpha]}(x)+2\omega_{l+1}\left(g;\sqrt{\Theta_{n,2}^{[\alpha]}}\right).
\end{eqnarray*}
\end{theorem}
\begin{proof}
From \cite{IEG}, for $0\leq x\leq l$ and $u\geq 0$, it holds
\begin{eqnarray*}
|g(u)-g(x)|\leq 4\mathcal{N}_f(1+x^2)(u-x)^2+\left(1+\frac{|u-x|}{\theta}\right)\omega_{l+1}(g;\theta),~~\theta>0. 
\end{eqnarray*} 
Now, applying the operators defined by (\ref{O1}) and with the help of Cauchy-Schwarz inequality, we can obtain
\begin{eqnarray*}
|\mathcal{U}_{n}^{[\alpha]}(g;x)-g(x)|&\leq & 4\mathcal{N}_f(1+x^2) \Theta_{n,2}^{[\alpha]}+\left(1+\frac{\mathcal{U}_{n}^{[\alpha]}(|u-x|)}{\theta}\right)\omega_{l+1}(g;\theta)\\
&\leq & 4\mathcal{N}_f(1+x^2) \Theta_{n,2}^{[\alpha]}(x)+\left(1+\frac{\sqrt{\Theta_{n,2}^{[\alpha]}} }{\theta}\right)\omega_{l+1}(g;\theta)\\
&\leq & 4\mathcal{N}_f(1+x^2) \Theta_{n,2}^{[\alpha]}(x)+\left(1 +1\right)\omega_{l+1}\left( g;\sqrt{\Theta_{n,2}^{[\alpha]}}\right)\\
&=& 4\mathcal{N}_f(1+x^2) \Theta_{n,2}^{[\alpha]}(x)+2\omega_{l+1}(g;\sqrt{\Theta_{n,2}^{[\alpha]}}).
\end{eqnarray*}

\end{proof}

\section{Quantitative Approximation}\label{sec4}
For estimations of the degree of approximation in the weighted space $C_w^k[0,\infty)$, Ispir \cite{IN1} proposed the weighted modulus of continuity $\Delta(g;\xi)$  for any $\xi>0$, 
as follows:

\begin{eqnarray}
\Delta(g;\xi)=\underset{0\leq h\leq\xi,~0\leq x\leq\infty}\sup \frac{|g(x+h)-g(x)|}{(1+h^2)(1+x^2)},~~~~~~~g\in C_w^k[0,\infty). 
\end{eqnarray} 

 \begin{remark}
For $g\in C_w^k[0,\infty)$
 \begin{eqnarray*}
  \underset{\xi\to 0}\lim\Delta(g;\xi)=0.
 \end{eqnarray*}
 \end{remark}
 On can obtains as, $\Delta(f;\lambda\xi)\leq2(1+\xi^2)(1+\lambda)\Delta(f;\xi),~~\lambda>0$. 
Using the  weighted modulus of continuity and defined inequality, one can show that
 \begin{eqnarray}
 \nonumber|g(t)-g(x)|&\leq &(1+x^2)(1+(t-x)^2)\Delta(g;|t-x|)\\
 &\leq & 2\left(1+\frac{|t-x|}{\xi}\right)(1+\xi^2)(1+(t-x)^2)(1+x^2)\Delta(f;|t-x|),~\text{for every} f\in C_w^k[0,\infty).
 \end{eqnarray}
As the consequence of the weighted modulus of continuity, we determine the degree of approximation of the operators $\mathcal{U}_{n}^{[\alpha]}(g;x)$ in the weighted space $C_w^k[0,\infty)$.
\subsection{Quantitative Voronovskaya type theorem}
\begin{theorem}
Let $\mathrm{g}', \mathrm{g}''\in C_w^k[0,\infty)$ and for sufficiently large value of $n\in\mathbb{N}$, then for each $x\geq0$, we yield
\begin{eqnarray*}
n\left |\mathcal{U}_{n}^{[\alpha]}(\mathrm{g};x)-\mathrm{g}(x)-\mathrm{g}'(x)\Theta_{n,1}^{[\alpha]}-\frac{\mathrm{g}''(x)}{2!}\Theta_{n,2}^{[\alpha]}\right|=O(1)\Lambda\left(f;\sqrt{\frac{1}{n}}\right).
\end{eqnarray*}
\end{theorem}

\begin{proof}
By Taylor's expansion, one can obtain
\begin{eqnarray}
\mathrm{g}(t)-\mathrm{g}(x)=\mathrm{g}'(x)(t-x)+\frac{\mathrm{g}''(x)}{2}(t-x)^2+\zeta(t,x),
\end{eqnarray}
where $\zeta(t,x)=\frac{\mathrm{g}''(\theta)-\mathrm{g}''(x)}{2!}(\theta-x)^2$ and $\zeta\in (t,x)$.
Applying operators (\ref{O1}) on both sides to above expansion, then one can obtains 
\begin{eqnarray}\label{n1}
n\left|\mathcal{U}_{n}^{[\alpha]}(\mathrm{g};x)-\mathrm{g}(x)-\mathrm{g}'(x)\Theta_{n,1}^{[\alpha]}-\frac{\mathrm{g}''(x)}{2}\Theta_{n,2}^{[\alpha]}\right|\leq n\mathcal{U}_{n}^{[\alpha]}(|\eta(t,x)|;x).
\end{eqnarray}
Now using the property of weighted modulus of continuity, we get
\begin{eqnarray*}
\frac{\mathrm{g}''(\theta)-\mathrm{g}''(x)}{2} 
&\leq & \left(1+\frac{|t-x|}{\xi} \right)(1+\xi^2)(1+(t-x)^2)(1+x^2)\Delta(f'',\xi)
\end{eqnarray*}
and also 
\begin{eqnarray}
\left|\frac{\mathrm{g}''(\theta)-\mathrm{g}''(x)}{2}\right| &\leq &  
\begin{cases}
    2(1+\xi^2)^2(1+x^2)\Delta(\mathrm{g}'',\xi),& |t-x|<\xi,\\
    2(1+\xi^2)^2(1+x^2)\frac{(t-x)^4}{\xi^4}\Delta(\mathrm{g}'',\xi),& |t-x|\geq\xi.
\end{cases} 
\end{eqnarray}
Now for $\xi\in(0,1)$, we get
\begin{eqnarray}
\left|\frac{\mathrm{g}''(\theta)-\mathrm{g}''(x)}{2}\right| &\leq & 8(1+x^2)\left(1+\frac{(t-x)^4}{\xi^4}\right)\Delta(\mathrm{g}'',\xi). 
\end{eqnarray}
Hence, $$(|\zeta(t,x)|;x)\leq 8(1+x^2)\left((t-x)^2+\frac{(t-x)^6}{\xi^4}\right)\Delta(\mathrm{g}'',\xi).$$
Thus, applying the Lemma \ref{l1}
\begin{eqnarray*}
\mathcal{U}_{n}^{[\alpha]}(| \zeta(t,x)|;x)&\leq & 8(1+x^2)\Delta(\mathrm{g}'',\xi)\left(\mathcal{U}_{n}^{[\alpha]}((t-x)^2;x)+\frac{\mathcal{U}_{n}^{[\alpha]}((t-x)^6;x)}{\xi^4}\right) \\
&\leq & 8(1+x^2)\Delta(\mathrm{g}'',\xi) \left(O\left(\frac{1}{n} \right)+\frac{1}{\xi^4} O\left(\frac{1}{n^3} \right) \right),~~\text{as}~n\to\infty.
\end{eqnarray*}
Choose, $\xi=\sqrt{\frac{1}{n}}$, then
\begin{eqnarray}
\mathcal{U}_{n}^{[\alpha]}(| \zeta(t,x)|;x)\leq 8 O\left(\sqrt{\frac{1}{n}} \right)\Delta\left(\mathrm{g}'',\sqrt{\frac{1}{n}}\right)(1+x^2).
\end{eqnarray}
Hence, we reach on 
\begin{eqnarray}\label{n2}
n\mathcal{U}_{n}^{[\alpha]}(|\zeta(t,x)|;x)=O(1)\Delta\left(\mathrm{g}'', \sqrt{\frac{1}{n}}\right).
\end{eqnarray}
By (\ref{n1}) and (\ref{n2}), we obtain the required result.
\end{proof}

\subsection{Gr$\ddot{\text{u}}$ss Voronovskaya type theorem}
\begin{theorem}\label{th2}
Let $\mathsf{f}\in C_w^k[0,\infty)$ then for $\mathsf{f}', \mathsf{f}'', \mathsf{g}', \mathsf{g}''\in C_w^k[0,\infty)$, it holds
\begin{eqnarray*}
\underset{n\to\infty}\lim n\left(\mathcal{U}_{n}^{[\alpha]}(\mathsf{f}\mathsf{g};x)-\mathcal{U}_{n}^{[\alpha]}(\mathsf{f};x)\mathcal{U}_{n}^{[\alpha]}(\mathsf{g};x) \right)= 2x \mathsf{f}'(x)\mathsf{g}'(x). 
\end{eqnarray*}
\end{theorem} 
\begin{proof}
By making suitable arrangement and using well known properties of derivative of multiplication of two functions, we get
\begin{eqnarray*}
 n\left(\mathcal{U}_{n}^{[\alpha]}(\mathsf{f}\mathsf{g};x)-\mathcal{U}_{n}^{[\alpha]}(\mathsf{f};x)\mathcal{U}_{n}^{[\alpha]}(\mathsf{g};x) \right)&=& n\Bigg\{\Bigg(\mathcal{U}_{n}^{[\alpha]}(\mathsf{f}\mathsf{g};x)-\mathsf{f}(x)\mathsf{g}(x)-(\mathsf{f}\mathsf{g})'\Theta_{n,1}^{[\alpha]}\\
 &&-\frac{(\mathsf{f}\mathsf{g})''}{2!}\Theta_{n,2}^{[\alpha]}\Bigg)-\mathsf{g}(x)\Bigg(\mathcal{U}_{n}^{[\alpha]}(\mathsf{f};x)-\mathsf{f}(x)\\
 &&-\mathsf{f}'(x)\Theta_{n,1}^{[\alpha]}-\frac{\mathsf{f}''(x)}{2!}\Theta_{n,2}^{[\alpha]} \Bigg)\\
 &&-\mathcal{U}_{n}^{[\alpha]}(\mathsf{f};x)\Bigg(\mathcal{U}_{n}^{[\alpha]}(\mathsf{g};x)-\mathsf{g}(x)-\mathsf{g}'(x)\Theta_{n,1}^{[\alpha]}\\
 &&-\frac{\mathsf{g}''(x)}{2!}\Theta_{n,2}^{[\alpha]} \Bigg)+\frac{\mathsf{g}''(x)}{2!}\mathcal{U}_{n}^{[\alpha]}((t-x)^2;x)\\
 &&\times \left(\mathsf{f}- \mathcal{U}_{n}^{[\alpha]}(\mathsf{f};x)\right)+\mathsf{f}'(x)\mathsf{g}'(x)\Theta_{n,2}^{[\alpha]}\\
 &&+ \mathsf{g}'(x)\Theta_{n,1}^{[\alpha]}\left(\mathsf{f}- \mathcal{U}_{n}^{[\alpha]}(\mathsf{f};x)\right) \Bigg\}.
\end{eqnarray*}
For sufficiently large value of $n$, i.e. for $n\to\infty$, $\alpha\to 0$. So with the help of Theorem \ref{th1} and \ref{th2} and talking the limit on both side of the above equation, we obtain the required result.  
\end{proof}

\section{Graphical and Numerical representation}\label{sec5}
This section consists, the graphical approcah and numerical analysis for the convergence of the operators to the function. 

\begin{example}
Consider the function is $g(x)=x^2\sin{\pi x}$ with $x\in[0,2]$ and choosing $\alpha=\frac{1}{60}$. Then the corresponding operators for $n=15,~35$ are  $\mathcal{U}_{15}^{[\alpha]}(g;x)$(green), $\mathcal{U}_{35}^{[\alpha]}(g;x)$(blue) respectively. One can observe that the better  rate of convergence can be obtained by graphical representation \ref{F1}, which is given below.  
\begin{figure}[h!]
    \centering 
    \includegraphics[width=.52\textwidth]{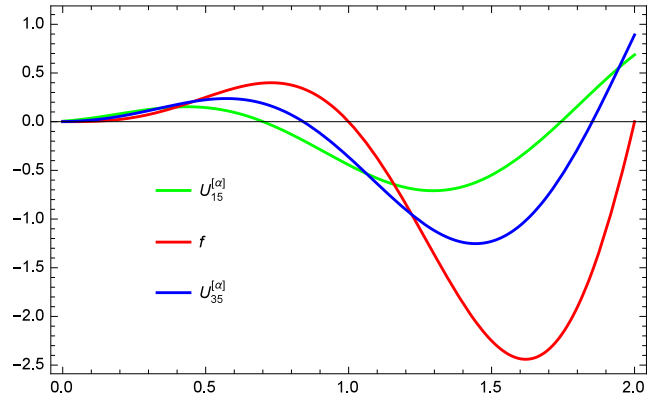}   
    \caption[Description in LOF, taken from~\cite{source}]{The convergence of the operators $\mathcal{U}_{n}^{[\alpha]}(g;x)$ to the function $f(x)(red)$.}
    \label{F1}
\end{figure}
\end{example}

\pagebreak
\begin{example}
Let us consider the function $f(x)=te^{-7x}$ (red) for which, the rate of convergence of the defined operators (\ref{O1}) is  discussed by taking different values of $n\in\mathbb{N}$. Choosing $n=5,10,20,25,35,40,45$, then the corresponding operators are represented by blue, green, cyan, brown, yellow, magenta, purple colors          receptively in the given Figure \ref{F2}. Here we take $\alpha=\frac{1}{45}$. 

\begin{figure}[h!]
    \centering 
    \includegraphics[width=.52\textwidth]{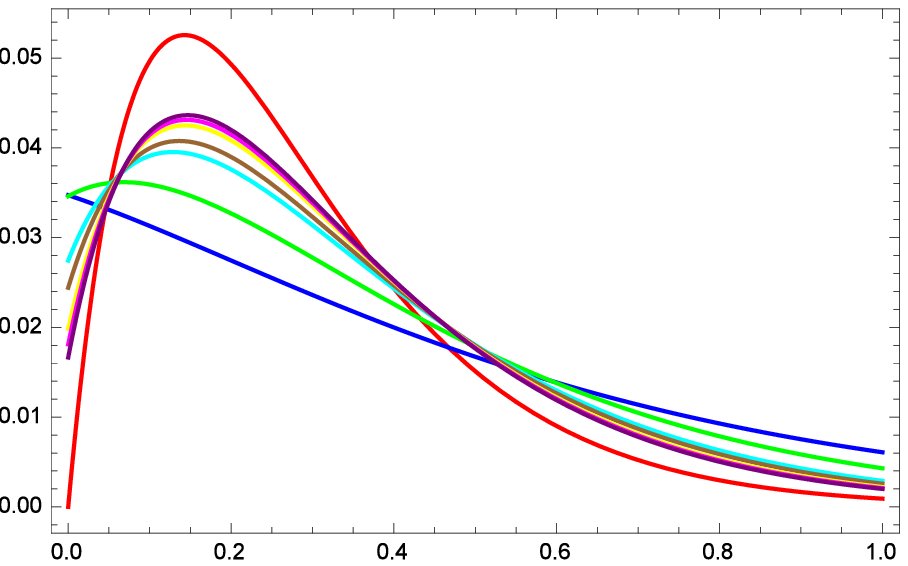}   
    \caption[Description in LOF, taken from~\cite{source}]{The convergence of the operators $\mathcal{U}_{n}^{[\alpha]}(f;x)$ to the function $f(x)(red)$.}
    \label{F2}
\end{figure}
\end{example}

\begin{example}
For the same function, which has been taken in the above example, we observe by changing the different values of $\alpha$ i.e. choosing  $\alpha=\frac{1}{10},\frac{1}{20},\frac{1}{40},\frac{1}{60},\frac{1}{80}$ for which the the corresponding operators for the fixed value of $n=10$ are represented by black, orange, pink, blue, green colors in given Figure \ref{F3}.
\begin{figure}[h!]
    \centering 
    \includegraphics[width=.52\textwidth]{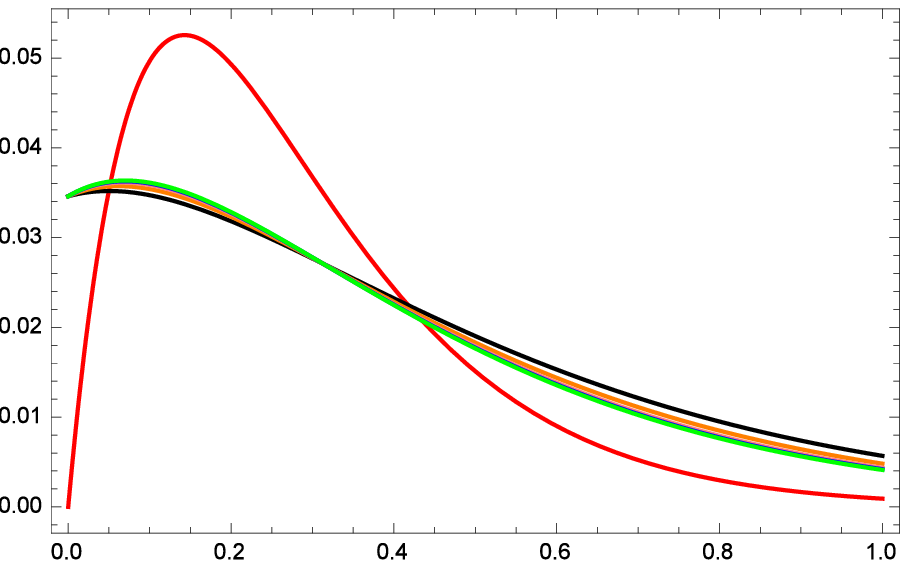}   
    \caption[Description in LOF, taken from~\cite{source}]{The convergence of the operators $\mathcal{U}_{n}^{[\alpha]}(f;x)$ to the function $f(x)(red)$.}
    \label{F3}
\end{figure}
\end{example}
\textbf{Concluding Remark:}
1. By the above Figures (\ref{F1},\ref{F2}), one can observe that as the value of $n$ is increased, the approximation is going to better while on taking the particular value of $\alpha$. i.e. by the suitable choice of $\alpha$, we can show the better approximation by taking large value of $n$. But in Figure \ref{F3}, convergence can be seen, when the  value of $n$ is fixed and the value of $\alpha$ is decreased.  \\

2. On choosing appropriate values of $\alpha$ and $n$, we can find the better approximation.

\begin{example}
Let the function  $f=(x^2+1)e^x$, take $$n=5,10,20,25,30,40,50,70,90,130,150,190,240,250,400,500$$
and $\alpha=\frac{1}{5},\frac{1}{10},\frac{1}{20},\frac{1}{30},\frac{1}{50},\frac{1}{100},\frac{1}{150},\frac{1}{200},\frac{1}{250},\frac{1}{500}$ then we obtain the approximation by given table.
\end{example}
\pagebreak
\begin{table}[ht]
\centering
\begin{tabular}{|c|c|c|c|c|c|c|c|c|c|c|}
\hline 
$n\downarrow$, $\alpha\to$ & $\frac{1}{5}$ & $\frac{1}{10}$ & $\frac{1}{20}$ & $\frac{1}{30}$ & $\frac{1}{50}$ & $\frac{1}{100}$ & $\frac{1}{150}$ & $\frac{1}{200}$ & $\frac{1}{250}$ & $\frac{1}{500}$\\  
\hline 
5 & 2.01244 & 1.85335 & 1.7963 & 1.77966 & 1.7671 & 1.75808 & 1.75515 & 1.7537 & 1.75283 & 1.7511 \\ 
\hline 
10 & - & 1.30625 & 1.28073 & 1.2732 & 1.26749 & 1.26338 & 1.26204 & 1.26137 & 1.26098 & 1.26019 \\ 
\hline 
20 & - & - & 1.13228 & 1.12711 & 1.12318 & 1.12034 & 1.11941 & 1.11896 & 1.11868 & 1.11814 \\ 
\hline
25 & - & - & - & 1.1034 & 1.09975 & 1.09711 & 1.09625 & 1.09582 & 1.09557 & 1.09506 \\ 
\hline 
30 & - & - & - & 1.08846 & 1.08497 & 1.08246 & 1.08164 & 1.08124 & 1.08099 & 1.08051 \\ 
\hline 
40 & - & - & - & - & 1.0674 & 1.06504 & 1.06427 & 1.06388 & 1.06366 & 1.0632 \\ 
\hline 
50 & - & - & - & - & 1.05732 & 1.05504 & 1.05429 & 1.05392 & 1.0537 & 1.05326 \\ 
\hline 
70 & - & - & - & - & - & 1.044 & 1.04328 & 1.04293 & 1.04272 & 1.0423 \\ 
\hline 
90 & - & - & - & - & - & 1.03804 & 1.03734 & 1.037 & 1.03679 & 1.03638 \\ 
\hline 
130 & - & - & - & - & - & - & 1.03108 & 1.03074 & 1.03054 & 1.03014 \\ 
\hline 
150 & - & - & - & - & - & - & 1.02923 & 1.0289 & 1.0287 & 1.0283 \\ 
\hline 
190 & - & - & - & - & - & - & - & 1.02639 & 1.02619 & 1.02579 \\ 
\hline 
240 & - & - & - & - & - & - & - & - & 1.02424 & 1.02385 \\ 
\hline 
250 & - & - & - & - & - & - & - & - & 1.02395 & 1.02356 \\ 
\hline
400 & - & - & - & - & - & - & - & - & - & 1.02093 \\ 
\hline  
500 & - & - & - & - & - & - & - & - & - & 1.02006 \\ 
\hline 
\end{tabular} 
\caption{Effect of $\alpha$ in the convergence of the operators $\mathcal{U}_{n}^{[\alpha]}(f;x)$}\label{t1}
\end{table}

\textbf{Observation:} By the above Table \ref{t1}, we can observe that as the value of $\alpha$ is decreased, the error is going to least for particular value of $n$. And the same time, if we see the table, we can observe that on increasing the values of $n$, the errors are decreased for particular value of $\alpha$ (excluding the `dash-'). These process is running on a particular point of $x$.

\section{A-Statistical Convergence of the defined operators}\label{sec6}
This section contains statistical convergence theorem. We establish some approximation property to study statistical convergence.  The basic idea of statistical convergence was first introduced by Fast \cite{FH} even though the first publication related to statistical convergence was in 1935 (published in Warsaw) and credit goes to Zygmund in his monograph also independently work is seen in paper Steinhaus \cite{SH} in 1951. In paper \cite{SIJ}, Schoenberg reintroduced the statistical convergence and now a days it has become an area of active research in approximation theory. \\

In 2002, it has been seen the use of statistical convergence in Approximation theory by Gadjiev \cite{GAD1}. We recall the symbol from \cite{q6}, as $A_{ni}=\{a_{ni}\}$ is infinite non-negative infinite  summability matrix. Here we denote $A$-transform of matrix $A_{ni}$ for a given sequence $\{x_i\}$ by
\begin{eqnarray*}
A_{ni} x_i=\sum\limits_{i=0}^\infty  a_{ni}x_i,
\end{eqnarray*}
provided $(A_{ni} x_i)$ converges for each $n\in\mathbb{N}$. Now if 
$\underset{n\to\infty}\lim(A_{ni} x_i)=\sigma$ whenever $\lim x_i=\sigma$ \cite{GH} then $A_{ni}$ is said to be regular. And also, $\underset{n\to\infty}\lim a_{ni}=0,~\forall~i\in\mathbb{N}$. In this case, 
$\{x_i\}$ is said to be  $A$-statistically convergent, i.e. for every $\epsilon>0$, $\underset{n\to\infty}\lim \sum_{\{i:||x_i-\sigma|\geq\epsilon\}}a_{ni}=0$  and it is written as $st_A-\lim x_i=0$. For more information, we refer to reader to see \cite{CJ,FAS,FJA,KE,MHI}.

\begin{theorem}
Consider $A_{ni}=\{a_{ni}\}$ be non-negative regular summability matrix and for each $f\in C_w^k[0,\infty)$ then for every $x\in[0,\infty)$, we have
\begin{eqnarray*}
st_A-\underset{n\to\infty}\lim \|\mathcal{U}_{n}^{[\alpha]}(f;x)-f(x)\|_{w(x)}=0.
\end{eqnarray*}
\end{theorem}

\begin{proof}
If we show
\begin{eqnarray}\label{eq6}
st_A-\underset{n\to\infty}\lim \|\mathcal{U}_{n}^{[\alpha]}(t^r;x)-x^r\|_{w(x)}=0~~~\text{for}~~r=0,1,2
\end{eqnarray} 
then we obtain the required result i.e. the proof will be done.

It is clear that
\begin{eqnarray*}
st_A-\underset{n\to\infty}\lim \|\mathcal{U}_{n}^{[\alpha]}(1;x)-1\|_{w(x)}=0.
\end{eqnarray*}
Also
\begin{eqnarray*}
 \|\mathcal{U}_{n}^{[\alpha]}(t;x)-x\|_{w(x)}&=&\underset{x\geq0}\sup\frac{1}{1+x^2} \frac{1}{n} \\
 &\leq & \frac{1}{n}.
\end{eqnarray*}
So we define the following sets for given $\epsilon>0$, as
\begin{eqnarray*}
\mathcal{V}_1&=&\{n:\|\mathcal{U}_{n}^{[\alpha]}(t;x)-x\|\geq \epsilon\}\\
\mathcal{V}_2&=&\left\{n:\frac{1}{n}\geq \epsilon\right\}.
\end{eqnarray*}
Obviously, $\mathcal{V}_1\subset\mathcal{V}_2$ and hence, $\sum_{i\in \mathcal{V}_1}a_{ni}\leq \sum_{i\in \mathcal{V}_2}a_{ni}$. Therefore
\begin{eqnarray*}
st_A-\underset{n\to\infty}\lim \|\mathcal{U}_{n}^{[\alpha]}(t;x)-x\|_{w(x)}=0.
\end{eqnarray*}
Further,
\begin{eqnarray*}
 \|\mathcal{U}_{n}^{[\alpha]}(t^2;x)-x^2\|_{w(x)}&=&\underset{x\geq0}\sup\frac{1}{1+x^2} \left(\frac{2+4nx+n^2x^2+n^2x\alpha}{n^2}-x^2 \right)\\
 &\leq &  \left(\frac{2}{n^2}+\frac{5}{2n} \right).
\end{eqnarray*}
Since R.H.S. of above inequality tends to zero as $n\to\infty$. So for given $\epsilon>0$, we can consider following sets, which show as
\begin{eqnarray*}
\mathcal{W}_1&=&\{n:\|\mathcal{U}_{n}^{[\alpha]}(t^2;x)-x^2\|\geq \epsilon\}\\
\mathcal{W}_2&=& \left\{n:\frac{2}{n^2}\geq \frac{\epsilon}{2}\right\}\\
\mathcal{W}_3&=&\left\{n:\frac{5}{2n}\geq \frac{\epsilon}{2}\right\},
\end{eqnarray*}
which implies that $\mathcal{W}_1\subset\mathcal{W}_2\cup\mathcal{W}_3 $, and hence $\sum_{i\in \mathcal{W}_1}a_{ni}\leq \sum_{i\in \mathcal{W}_2}a_{ni}+\sum_{i\in \mathcal{W}_3}a_{ni}$. Therefore
\begin{eqnarray*}
st_A-\underset{n\to\infty}\lim \|\mathcal{U}_{n}^{[\alpha]}(t^2;x)-x^2\|_{w(x)}=0.
\end{eqnarray*}
Thus the the proof is completed.
\end{proof}

\begin{corollary}
Let $w_\zeta(x)\geq 1$ be a continuous function such that $\underset{|x|\to\infty}\lim\frac{w(x)}{w_\zeta(x)}\to 0$ then for each $f\in C_w^k[0,\infty)$ with $\{a_{ni}\}$ be non-negative regular summability matrix, we have
\begin{eqnarray*}
st_A-\underset{n\to\infty}\lim \|\mathcal{U}_{n}^{[\alpha]}(t^2;x)-x^2\|_{w_\zeta(x)}=0.
\end{eqnarray*}
\end{corollary}

By means of Peetre’s $K$-functional, $A$-Statistical convergence is introduced for the operators $\mathcal{U}_{n}^{[\alpha]}$ in next theorem.

\begin{theorem}
Let $g\in C_B^2[0,\infty)$ and for every $x\in[0,\infty)$, we have
\begin{eqnarray*}
st_A-\underset{n\to\infty}\lim \|\mathcal{U}_{n}^{[\alpha]}(g;x)-g\|_{C_B[0,\infty)}=0,~~~\forall~n\in\mathbb{N}.
\end{eqnarray*}
\end{theorem}

\begin{proof}
Using expansion of Taylor, we have
\begin{eqnarray*}
g(t)=g(x)+g'(x)(t-x)+\frac{1}{2}g''(\tau)(t-x)^2,
\end{eqnarray*}
where $\tau\in [t,x]$. Now applying $\mathcal{U}_{n}^{[\alpha]}$, we obtain
\begin{eqnarray*}
\mathcal{U}_{n}^{[\alpha]}(g(t)-g(x);x)=f'(x)\Theta_{n,1}^{[\alpha]}(x)+\frac{g''(\tau)}{2}\Theta_{n,2}^{[\alpha]}(x).
\end{eqnarray*}
In this way, 
\begin{eqnarray}\label{eq7}
\|\mathcal{U}_{n}^{[\alpha]}(g(t)-g(x);x)\|_{C_B[0,\infty)}&\leq &\|g'(x)\|_{C_B[0,\infty)}\|\Theta_{n,1}^{[\alpha]}(x)\|_{C_B[0,\infty)}+\|g''(\tau)\|_{C_B[0,\infty)}\|\Theta_{n,2}^{[\alpha]}(x)\|_{C_B[0,\infty)}\nonumber\\
&=&\mathcal{E}_1+\mathcal{E}_2, ~(\text{say})
\end{eqnarray}
From (\ref{eq6}), we can write
\begin{eqnarray*}
\underset{n\to\infty}\lim\sum_{\{i\in\mathbb{N} \mathcal{E}_1\geq\frac{\epsilon}{2}\}}a_{ni}&=&0,\\
\underset{n\to\infty}\lim\sum_{\{i\in\mathbb{N}: \mathcal{E}_2\geq\frac{\epsilon}{2}\}}a_{ni}&=&0.
\end{eqnarray*}
So by equation (\ref{eq7}), we get
\begin{eqnarray*}
\underset{n\to\infty}\lim\sum_{\{i\in\mathbb{N}: \|\mathcal{U}_{n}^{[\alpha]}(f(t)-f(x);x)\|_{C_B[0,\infty)}\geq\epsilon\}}a_{ni}\leq \underset{n\to\infty}\lim\sum_{\{i\in\mathbb{N}: \mathcal{E}_1\geq\frac{\epsilon}{2}\}}a_{ni}+\underset{n\to\infty}\lim\sum_{\{i\in\mathbb{N}: \mathcal{E}_2\geq\frac{\epsilon}{2}\}}a_{ni}.
\end{eqnarray*}
And hence, 
\begin{eqnarray*}
st_A-\underset{n\to\infty}\lim \|\mathcal{U}_{n}^{[\alpha]}(g;x)-g\|=0.
\end{eqnarray*}

Hence proved.
\end{proof}
\begin{theorem}
Consider $f\in C_B[0,\infty)$ and for each $n\in\mathbb{N}$, an inequality holds
\begin{eqnarray*}
\|\mathcal{U}_{n}^{[\alpha]}(f(t)-f(x);x)\|_{C_B[0,\infty)}\leq \mathcal{C}\omega_2(f;\sqrt{\xi}).
\end{eqnarray*}
\end{theorem}
\begin{proof}
Consider a function $h\in C_B^2[0,\infty)$. We can write
\begin{eqnarray*}
\|\mathcal{U}_{n}^{[\alpha]}(h(t)-h(x);x)\|_{C_B[0,\infty)}&\leq & \|h'\|_{C_B[0,\infty)} \|\mathcal{U}_{n}^{[\alpha]}((t-x);x)\|_{C_B[0,\infty)}+\frac{1}{2}\|h''\|_{C_B[0,\infty)}\|\Theta_{n,2}^{[\alpha]}\|_{C_B[0,\infty)}\\
&\leq & \eta \|h\|_{C_B^2[0,\infty)}. 
\end{eqnarray*}
Since, $f\in C_B[0,\infty)$ and $h\in C_B^2[0,\infty)$, therefore, by above inequality, one can write
\begin{eqnarray*}
\|\mathcal{U}_{n}^{[\alpha]}(f(t)-f(x);x)\|_{C_B[0,\infty)}&\leq & \|\mathcal{U}_{n}^{[\alpha]}(f;x)-\mathcal{U}_{n}^{[\alpha]}(h;x)\|_{C_B[0,\infty)} + \|\mathcal{U}_{n}^{[\alpha]}(h(t)-h(x);x)\|_{C_B[0,\infty)}\\
&&\|h-f \|_{C_B[0,\infty)}\\
&\leq & 2\|h-f \|_{C_B[0,\infty)}+\|\mathcal{U}_{n}^{[\alpha]}(h(t)-h(x);x)\|_{C_B[0,\infty)}\\
&\leq & 2\|h-f \|_{C_B[0,\infty)}+\eta \|h\|_{C_B^2[0,\infty)}.
\end{eqnarray*}
Using the property of Peetre’s $K$-functional (\ref{pe1}), one can has
\begin{eqnarray*}
\|\mathcal{U}_{n}^{[\alpha]}(f(t)-f(x);x)\|_{C_B[0,\infty)}&\leq & \mathcal{C}\{\omega_2(f;\sqrt{\eta})+\min(1,\eta)\|f\|_{C_B[0,\infty)} \}.
\end{eqnarray*}
Using (\ref{eq6}), we obtain
$\eta\to 0$ statistically as $n\to\infty$.

In this way, 
$\omega_2(f;\sqrt{\xi})\to 0$ statistically as $n\to\infty$.

Hence, the rate of convergence via $A$-statistically is obtained of the sequence of linear positive operators defined by \ref{O1} to the function $f(x)$.
\end{proof}

\begin{remark}
If $A=I$ then the ordinary rate of convergence is obtained. 

\end{remark}


\section{Rate of convergence by means of function of bounded variation}\label{sec7}
In this section, the rate of convergence of the said operators is determined in the space of the functions with derivative of bounded variation. Here, we consider $DBV[0,\infty)$, the set of all continuous function having derivative of bounded variation on every finite sub-interval of the $[0,\infty)$. 
On observing that for each $g\in DBV[0,\infty)$ and $a>0$, one can write
\begin{eqnarray}
g(x)=\int\limits_a^x h(s)~ds+g(a),
\end{eqnarray}
where $h$ is a function bounded variation on each finite sub-interval of $[0,\infty)$. 
Here, we use an auxiliary operators $g_x$ for every $g\in DBV[0,\infty)$ for obtaining the rate of of convergence of the proposed operators.

\begin{eqnarray}\label{eq5}
g_x(t) &= &  
\begin{cases}
    g(t)-g(x-),& 0\leq t<x,\\
    0,& t=x,\\
    g(t)-g(x+), &  x<t<\infty.
\end{cases} 
\end{eqnarray}
Generally, we denote $V_a^b \mathfrak{f}$ is total variation of a real valued function $g$ defined on $[a,b]\subset[0,\infty)$ with the quantity
\begin{eqnarray}
V_a^b \mathfrak{f}=\underset{\mathcal{P}}\sup\left(\sum\limits_{k=0}^{n_{P}-1}|g(x_{k+1})-f(x_k)| \right),
\end{eqnarray} 
where $\mathcal{P}$ is the set of all partition $P=\{a=x_0,\cdots,x_{n_P}=b\}$ of the interval $[a,b]$. 

\begin{lemma}\label{l3}
For sufficiently large value of $n$, for every $x\geq 0$ then there exist a positive constant $M>0$ such that
\begin{enumerate}
\item{} $h_n^{[\alpha]}(x,y)= \int\limits_0^y u_{n}^{[\alpha]}(x,t)~dt\leq \frac{3\eta_n^2(x)}{n(x-y)^2},~~0\leq y<x,$
\item{} $1-h_n^{[\alpha]}(x,z)= \int\limits_z^\infty u_{n}^{[\alpha]}(x,t)~dt\leq \frac{3\eta_n^2(x)}{n(z-x)^2},~~x<z<\infty.$
\end{enumerate}
\end{lemma}
\begin{proof}
For $y\in[0,x)$, it is holds
\begin{eqnarray*}
\int\limits_0^y u_{n}^{[\alpha]}(x,t)~dt 
&\leq &\frac{1}{(x-y)^2}\mathcal{U}_{n}^{[\alpha]}((t-x)^2;x)\\
&\leq & \frac{3\eta_n^2(x)}{n(x-y)^2}.
\end{eqnarray*}
Similarly, other result can be proved. 
\end{proof}
\begin{theorem}
Let $g\in DBV[0,\infty)$, for sufficiently large value of $n$ and $\max \alpha=\frac{1}{n}$. If $g(t)=O(t^s)$ as $t\to\infty$, then for $x\in[0,\infty)$, we obtain
\begin{eqnarray*}
|\mathcal{U}_{n}^{[\alpha]}(g;x)-g(x)|&\leq & \frac{3\eta_n^2(x)}{nx^2} |(g(2x)-g(x)-xg'(x+))| +\frac{x}{\sqrt{n}} V_x^t(g_x') + \frac{3\eta_n^2(x)}{n}\sum_{l=1}^{[\sqrt{nx}]} V_x^{x+\frac{x}{l}}(g_x') + \mathcal{M}_{s,x}^\gamma\\
&&+\frac{3|g(x)|\eta_n^2(x)}{nx^2} + |g'(x+)| \sqrt{\frac{3}{n}}\eta_n(x)+\frac{1}{2}\sqrt{\frac{3}{n}}\{g'(x+)-g'(x-)\}\eta_n(x)+  \frac{1}{2n}\{g'(x+)-g'(x-)\},
\end{eqnarray*}
where
\begin{eqnarray*}
\mathcal{M}_{s,x}^\gamma=M 2^{\gamma} \left(\int_{0}^\infty (t-x)^{2s} u_{n}^{[\alpha]}(x,t)dt\right)^{\frac{\gamma}{2s}}.
\end{eqnarray*}

\end{theorem}

\begin{proof}
We have 
\begin{eqnarray*}
|\mathcal{U}_{n}^{[\alpha]}(g;x)-g(x)|=\int_0^\infty u_{n}^{[\alpha]}  (x,t)(g(t)-g(x)~dt =\int_0^\infty u_{n}^{[\alpha]}  (x,t) \left(\int_0^tg'(u)~du\right).
\end{eqnarray*}
Since $g\in DBV[0,\infty)$, one can write an identity
\begin{eqnarray}
\nonumber g'(v)&=&\frac{1}{2}\{g'(x+)+g'(x-)\}+g_x'(v)+\frac{1}{2}\{g'(x+)-g'(x-)\}\text{sgn}(v-x)\nonumber\\
&&+\xi_x(v)\left(g'(v)- \frac{1}{2}(g'(x+)+g'(x-))\right), 
\end{eqnarray}
where
\begin{eqnarray}
\xi_x(v)=
\begin{cases}
1 & v=x\\
0 & v\neq x.
\end{cases}
\end{eqnarray}
Now, one can easily obtain as 
\begin{eqnarray*}
\int_0^\infty u_{n}^{[\alpha]}(x,t) \left(\int_0^t \left( \xi_x(v)\left(g'(v)- \frac{1}{2}(g'(x+)+g'(x-))\right)dv\right)\right)dt=0
\end{eqnarray*}
Also, 
\begin{eqnarray}
\end{eqnarray}
\begin{eqnarray}
\int_0^\infty u_{n}^{[\alpha]}(x,t) \left(\int_0^t\frac{1}{2}\{g'(x+)-g'(x-)\}\text{sgn}(v-x)dv\right)dt &\leq &\frac{1}{2}\{g'(x+)-g'(x-)\}\mathcal{U}_{n}^{[\alpha]}(|t-x|;x)\nonumber\\
&\leq & \frac{1}{2}\{g'(x+)-g'(x-)\}\left(\mathcal{U}_{n}^{[\alpha]}((t-x)^2;x)\right)^{\frac{1}{2}},
\end{eqnarray}
and
\begin{eqnarray}
\int_0^\infty u_{n}^{[\alpha]}(x,t) \left(\int_0^t\frac{1}{2}\{g'(x+)-g'(x-)\}dv\right)dt= \frac{1}{2}\{g'(x+)-g'(x-)\}\mathcal{U}_{n}^{[\alpha]}((t-x);x).
\end{eqnarray}
So, we have
\begin{eqnarray}\label{in3}
|\mathcal{U}_{n}^{[\alpha]}(g;x)-g(x)| &\leq & \left| \int_0^x\left(\int_0^t g_x'(v) dv\right)u_{n}^{[\alpha]}(x,t) dt+\int_x^\infty \left(\int_0^t g_x'(v) dv\right)u_{n}^{[\alpha]}(x,t) dt \right|\nonumber\\
&& +\frac{1}{2}\{g'(x+)-g'(x-)\}\left(\mathcal{U}_{n}^{[\alpha]}((t-x)^2;x)\right)^{\frac{1}{2}}\nonumber\\
&& + \frac{1}{2}\{g'(x+)-g'(x-)\}\mathcal{U}_{n}^{[\alpha]}((t-x);x)\nonumber\\
&\leq & E_{nx} + F_{nx} +\frac{1}{2}\sqrt{\frac{3}{n}}\{g'(x+)-g'(x-)\}\eta_n(x)+  \frac{1}{2n}\{g'(x+)-g'(x-)\},
\end{eqnarray}
where
\begin{eqnarray*}
E_{nx}=\left| \int_0^x\left(\int_0^t g_x'(v) dv\right)u_{n}^{[\alpha]}(x,t) dt \right|,
\end{eqnarray*}
and
\begin{eqnarray}
F_{nx}=\left| \int_x^\infty \left(\int_0^t g_x'(v) dv\right)u_{n}^{[\alpha]}(x,t) dt \right|.
\end{eqnarray}
Now applying Lemma \ref{l3}, and let $x=\frac{y\sqrt{n}}{\sqrt{n}-1}$, then integrating by part of 
\begin{eqnarray*}
E_{nx}&=&\left| \int_0^x \left(\int_0^t g_x'(v) dv\right) d_t(h_n^{[\alpha]}(x,t)) \right|\\
&=& \left| \int_0^x (h_n^{[\alpha]}(x,t)g_x'(t) dt  \right|\\
&\leq & \int_0^y |h_n^{[\alpha]}(x,t)| |g_x'(t)|dt+\int_y^x |h_n^{[\alpha]}(x,t)| |g_x'(t)|dt\\
&=& \int_0^{x-\frac{x}{\sqrt{n}}} h_n^{[\alpha]}(x,t) |g_x'(t)|dt +
\int_{x-\frac{x}{\sqrt{n}}}^x h_n^{[\alpha]}(x,t) |g_x'(t)|dt.
\end{eqnarray*}
With the help of (\ref{eq5}) and using Lemma \ref{l3}, we get 
\begin{eqnarray}
\int_{x-\frac{x}{\sqrt{n}}}^x h_n^{[\alpha]}(x,t) |g_x'(t)|dt &\leq & \int_{x-\frac{x}{\sqrt{n}}}^x |g_x'(t)-g_x'(x)|dt \nonumber\\
&\leq & \frac{x}{\sqrt{n}}V_{x-\frac{x}{\sqrt{n}}}^x.
\end{eqnarray}
Also, let $v=1+\frac{t}{x-t}$ and using Lemma \ref{l3}, we obtain 
\begin{eqnarray}
\int_0^{x-\frac{x}{\sqrt{n}}} h_n^{[\alpha]}(x,t) |g_x'(t)|dt &=& \frac{3\eta_n^2(x)}{n}\int_0^{x-\frac{x}{\sqrt{n}}}\frac{|g_x'(t)|}{(x-t)^2} dt= \frac{3\eta_n^2(x)}{n}\int_1^{\sqrt{n}} V_{x-\frac{x}{v}}^x (g_x') dv \nonumber\\&\leq & \frac{3\eta_n^2(x)}{n}\sum_{l=1}^{[\sqrt{n}]}V_{x-\frac{x}{l}}^x (g_x').
\end{eqnarray}
Hence
\begin{eqnarray}
E_{nx}\leq \frac{x}{\sqrt{n}}V_{x-\frac{x}{\sqrt{n}}}^x+\frac{3\eta_n^2(x)}{n}\sum_{l=1}^{[\sqrt{n}]}V_{x-\frac{x}{l}}^x (f_x').
\end{eqnarray}
Now another part can be written as 
\begin{eqnarray}
F_{nx}&=& \left| \int_x^\infty \left(\int_0^t g_x'(v) dv\right)u_{n}^{[\alpha]}(x,t) dt \right|\nonumber\\
&\leq & \left|\int_x^{2x} \left(\int_0^t g_x'(v) dv\right) d_t(1-h_n^{[\alpha]}(x,t))\right|+\left| \int_{2x}^\infty \left(\int_0^t g_x'(v) dv\right)u_{n}^{[\alpha]}(x,t) dt \right| \nonumber\\
&=& \left| \int_x^{2x} g_x'(v) dv (1-h_n^{[\alpha]}(x,2x))-\int_x^{2x} g_x'(t)(1-h_n^{[\alpha]}(x,t)) dt \right|\nonumber\\
&&+\left| \int_{2x}^\infty \left(\int_0^t (g'(v)-g'(x+)) dv\right)u_{n}^{[\alpha]}(x,t) dt \right|\nonumber\\
&\leq & \left| \int_x^{2x} (g'(v)-g'(x+)) dv (1-h_n^{[\alpha]}(x,2x))\right|+\left| \int_x^{2x} g_x'(t)(1-h_n^{[\alpha]}(x,t)) dt \right|\nonumber\\
&&+ \int_{2x}^\infty |(g(t)-g(x))| u_{n}^{[\alpha]}(x,t)dt +\int_{2x}^\infty |g'(x+)||(t-x)|u_{n}^{[\alpha]}(x,t)dt \nonumber\\
&\leq & \frac{3\eta_n^2(x)}{nx^2} |(g(2x)-g(x)-xg'(x+))|+\int_{x}^{x+\frac{x}{\sqrt{n}}}|g_x'(t)||(1-h_n^{[\alpha]}(x,t))| dt\nonumber\\
&& + \int_{x+\frac{x}{\sqrt{n}}}^{2x}|g_x'(t)||(1-h_n^{[\alpha]}(x,t))| dt + \int_{2x}^\infty M t^\gamma u_{n}^{[\alpha]}(x,t)dt\nonumber\\
&&+\int_{2x}^\infty |g(x)| u_{n}^{[\alpha]}(x,t)dt+ |g'(x+)| \int_{0}^\infty u_{n}^{[\alpha]}(x,t)|t-x| dt \nonumber\\
&\leq & \frac{3\eta_n^2(x)}{nx^2} |(g(2x)-g(x)-xg'(x+))|+ \int_{x}^{x+\frac{x}{\sqrt{n}}} V_x^t(g_x') dt+ \frac{3\eta_n^2(x)}{n} \int_{x+\frac{x}{\sqrt{n}}}^{2x} \frac{V_x^t(g_x')}{(x-t)^2}dt\nonumber\\
&&+M \int_{2x}^\infty t^\gamma u_{n}^{[\alpha]}(x,t)dt + |g(x)|\int_{2x}^\infty  u_{n}^{[\alpha]}(x,t)dt+ |g'(x+)| \left(\int_{0}^\infty u_{n}^{[\alpha]}(x,t)(t-x)^2 dt\right)^{\frac{1}{2}}\nonumber\\
&&\times \left(\int_{0}^\infty u_{n}^{[\alpha]}(x,t) dt\right)^{\frac{1}{2}}\nonumber\\
&\leq &  \frac{3\eta_n^2(x)}{nx^2} |(g(2x)-g(x)-xf'(x+))|+ \frac{x}{\sqrt{n}} V_x^t(f_x')+ \frac{3\eta_n^2(x)}{n} \int_{x+\frac{x}{\sqrt{n}}}^{2x} \frac{V_x^t(f_x')}{(x-t)^2}dt\nonumber\\
&&+M \int_{2x}^\infty t^\gamma u_{n}^{[\alpha]}(x,t)dt + |f(x)|\int_{2x}^\infty  u_{n}^{[\alpha]}(x,t)dt+ |f'(x+)| \sqrt{\frac{3}{n}}\eta_n(x) \nonumber\\
&\leq & \frac{3\eta_n^2(x)}{nx^2} |(g(2x)-g(x)-xg'(x+))| +\frac{x}{\sqrt{n}} V_x^t(g_x') + \frac{3\eta_n^2(x)}{n}\sum_{l=1}^{[\sqrt{nx}]} V_x^{x+\frac{x}{l}}(g_x')\nonumber\\
&&+M \int_{2x}^\infty t^\gamma u_{n}^{[\alpha]}(x,t)dt + |g(x)|\int_{2x}^\infty  u_{n}^{[\alpha]}(x,t)dt+ |g'(x+)| \sqrt{\frac{3}{n}}\eta_n(x).
\end{eqnarray}
Here, it is arising a case $t\leq 2(t-x)$ and $x\leq t-x$, when $t\geq 2x$, then by using H$\ddot{\text{o}}$lder inequality, we can obtain
\begin{eqnarray*}
F_{nx}&\leq &\frac{3\eta_n^2(x)}{nx^2} |(g(2x)-g(x)-xg'(x+))| +\frac{x}{\sqrt{n}} V_x^t(g_x') + \frac{3\eta_n^2(x)}{n}\sum_{l=1}^{[\sqrt{nx}]} V_x^{x+\frac{x}{l}}(g_x') \\
&& + M 2^{\gamma} \left(\int_{0}^\infty (t-x)^{2s} u_{n}^{[\alpha]}(x,t)dt\right)^{\frac{\gamma}{2s}}+ \frac{3|g(x)|\eta_n^2(x)}{nx^2} + |g'(x+)| \sqrt{\frac{3}{n}}\eta_n(x)\\
&=& \frac{3\eta_n^2(x)}{nx^2} |(g(2x)-g(x)-xg'(x+))| +\frac{x}{\sqrt{n}} V_x^t(g_x') + \frac{3\eta_n^2(x)}{n}\sum_{l=1}^{[\sqrt{nx}]} V_x^{x+\frac{x}{l}}(g_x') + \mathcal{M}_{s,x}^\gamma\\
&&+\frac{3|g(x)|\eta_n^2(x)}{nx^2} + |g'(x+)| \sqrt{\frac{3}{n}}\eta_n(x).
\end{eqnarray*}
Using the values of $E_{nx}$ and $F_{nx}$ in equation (\ref{in3}), we get our required result.
\end{proof}

\section{Conclusion and result discussion}
After whole discussions of the present article, it can be seen that the results are good regarding approximations for the defined operators. We discussed those properties which define the order of approximation in terms of modulus of continuity also by means of modified  Lipschitz type space. Steklov function is also one of the best useful function by which the present article deals the approximation property. Some results have been discussed in the weighted spaces which enrich the quality of our works. Quantitative approximation have been studied and asymptotic behavior of the operators, Gr$\ddot{\text{u}}$ss Voronovskaya type theorem have also been discussed quantitatively. As for supporting of the proof of convergence of the said operators, examples took place. An important property which is statistical convergence of the operators, has been established and statistical rate of convergence obtained. At last, the very beautiful property has been discussed and that is the rate of convergence the term of function with derivative of bounded variation. This research article may found to be useful in the area of analysis for the researchers. And results can be applicable in literature on Mathematical Analysis, Applied Mathematics.

\section{Applications}
As an application, the defined operators can be used in Quantum Calculus, Mathematical physics. Moreover, it can be generalized into complex number, also it can be obtained the better rate of convergence by using King's approach. A consequential research topic is approximation of the function by positive linear operators in general mathematics and it withal provides potent implements to application in areas of CAGD, numerical analysis, and solutions of differential equations. These operators can be generalized by considering hypergeometric function. It can be quite effective for the rate of convergence  while using $q$-integer in quantum calculus.

\end{document}